\documentclass{amsart}
\usepackage{amssymb}
\usepackage{amsmath}
\usepackage{url}
\usepackage{booktabs}
\usepackage{textcomp}
\usepackage{nicefrac}
\newcommand{\C}{\mathbf{C}}
\newcommand{\HH}{\mathbf{H}}

\newcommand{\Z}{\mathbf{Z}}
\newcommand{\R}{\mathbf{R}}
\def\Im{\operatorname{Im}}
\newcommand{\seq}{\medspace=\medspace}
\newcommand{\sle}{\medspace\le\medspace}
\newcommand{\slt}{\medspace < \medspace}
\newcommand{\sge}{\medspace\ge\medspace}
\newcommand{\splus}{\medspace+\medspace}
\newcommand{\sminus}{\medspace-\medspace}

\newtheorem{theorem}{Theorem}
\newtheorem*{thm1}{Theorem~1}

\newtheorem{corollary}[theorem]{Corollary}
\newtheorem{conjecture}[theorem]{Conjecture}
\newtheorem{remark}[theorem]{Remark}
\newtheorem{lemma}[theorem]{Lemma}

\begin{document}

\title[An explicit height bound for the classical modular polynomial]{An explicit height bound for the\\classical modular polynomial}

\author{Reinier Br\"oker and Andrew V. Sutherland}
\address{Brown University, Box 1917, 151 Thayer Street, Providence, RI}
\email{reinier@math.brown.edu}
\address{Massachusetts Institute of Technology, 77 Massachusetts Avenue, Cambridge, MA}
\email{drew@math.mit.edu}

\subjclass[2000]{Primary 11G05}

\begin{abstract}
For a prime $l$, let $\Phi_l$ be the classical modular polynomial, and let $h(\Phi_l)$ denote its logarithmic height.
By specializing a theorem of Cohen, we prove that $h(\Phi_l)\le 6l\log l + 16l + 14\sqrt{l}\log l$.  As a corollary,
we find that $h(\Phi_l)\le 6l\log l + 18l$ also holds.
A table of $h(\Phi_l)$ values is provided for $l \le 3607$.
\end{abstract}

\maketitle
\section{Introduction}
Let $j:\HH\to\C$ be the elliptic modular function.
For a positive integer $m$, the classical modular polynomial $\Phi_m$ is the minimal polynomial of the function $j(mz)$ over the field $\C(j)$.
As a polynomial in two variables, we have $\Phi_m(j(mz),j(z))=0$.
The polynomial $\Phi_m(X,Y)$ is symmetric, with integer coefficients, and has degree $\psi(m)=m\prod_{p|m}(1+p^{-1})$ in both variables \cite[Ch.~5]{Lang:EllipticFunctions}.

For a nonzero polynomial $P$ with complex coefficients, let $h(P)=\log\medspace\max|c|$, where $c$ ranges over the coefficients of $P$.
Cohen proves in \cite{CohenPaula:ModularPolynomials} that the bound
$$
h(\Phi_m) = 6\psi(m)\bigl(\log m - 2\kappa(m) + O(1)\bigr)
$$
holds, with $\kappa(m)=\sum_{p|m}p^{-1}\log p=O(\log\log m)$.  For $m=l$ prime, this yields
$$
h(\Phi_l) = 6l\log l + O(l).
$$
The purpose of this paper is to prove an explicit upper bound on the $O(l)$ term.
\begin{theorem}
For every prime $l$ we have $h(\Phi_l)\sle 6l\log l \splus 16l \splus 14\sqrt{l}\log l$.
\end{theorem}
\noindent
Such an explicit bound is needed to obtain rigorous results from algorithms that compute $\Phi_l$, including \cite{BrokerLauterSutherland:CRTModPoly,CharlesLauter:ModPoly,Elkies:AtkinBirthday,Enge:ModularPolynomials}.
To prove Theorem 1, we retrace the proof of Cohen, specializing to the case that $m=l$ is prime and seeking only an upper bound.
These restrictions simplify our presentation, but the core of the argument remains the same.
Our main contribution is to make the bound fully explicit.

We also add two refinements to the proof,
Lemma~5 and Lemma~8, that sharpen the constants we obtain in Theorem~1.  These improvements have a significant impact in the range of $l$ practical for computation.

\section{Bounding the height}
We recall that the modular polynomial may be written as
\begin{equation}\label{modular_equation}
\Phi_l(X,j(z)) = \bigl(X-j(lz)\bigr)\prod_{b=0}^{l-1}\bigl(X-j((z+b)/l)\bigr),
\end{equation}
see \cite[Ch.~5.2]{Lang:EllipticFunctions}.  To simplify our notation, for any $y\in\C$, let $\Phi_{l,y}(X)$ denote the univariate polynomial $\Phi_l(X,y)$, and let us define
$\bigl\|y\bigr\|=\log\medspace\max(1,|y|)$.
\begin{lemma}\label{lemma:hlybound}
Let $l$ be prime and let $z\in\HH$.  For $y=j(z)$ we have
$$
h(\Phi_{l,y}) \sle \bigl\|j(lz)\bigr\| + \sum_{b=0}^{l-1}\medspace\bigl\|j((z+b)/l)\bigr\| + (\log 2)l.
$$
\end{lemma}
\begin{proof}
Each coefficient $c$ of a monic polynomial with roots $\omega_1,\ldots,\omega_n\in\C$ 
satisfies $|c|\le 2^{n-1}\prod\max(1,|\omega_k|)$.
Applying this to (\ref{modular_equation}) yields the lemma.
\end{proof}

To bound $h(\Phi_l)$, it suffices to bound $h(\Phi_{l,y})$ for real~$y$ in an interval $[L,2L]$.

\begin{lemma}\label{lemma:interpolation}
Assume $h(\Phi_{l,y})\le B$ for $y\in[1728,3456]$.  Then
$$
h(\Phi_l)\le B + 2.083(l+1).
$$
\end{lemma}
\begin{proof}
This follows immediately from Lemma~\ref{lemma:A4} of the appendix.
\end{proof}

As explained in \cite[p.~397]{CohenPaula:ModularPolynomials}, the real value $j(it)$ increases 
monotonically with the real parameter $t\ge 1$. 
We have $j(i)=1728$, and for each $y$ in the interval $[1728,3456]$ there is a unique $t$ for 
which $y=j(it)$.  We find that $j(1.254i)>3456$, hence for
$y \in[1728,3456]$ we have $t < 1.254$.  Our main task is to bound the sum
\begin{equation}\label{Sl1}
S(l,t) = \sum_{b=0}^{l-1}\medspace\bigl\|j((it+b)/l)\bigr\|,
\end{equation}
as a function of $l$, for any fixed $t\in [1,1.254)$.  To do so, we use the following lemma.

\begin{lemma}\label{lemma:jbound}
Let $z\in\HH$ and let $x=\Im z$.  We have the bound
$$\bigl\|j(z)\bigr\| \sle 2\pi x + c(x),$$
where $c(x)\seq\log(e^{2\pi\max(x,1/x)}+1728-e^{2\pi})-2\pi x \slt 9.429$ for all $x\ge 1/2$.
\end{lemma}
\begin{proof}
From \cite[Lemma 1]{Brisebarre:jFunctionCoefficients} we have
$$
|j(z)|\sle j(ix)\sle e^{2\pi\max(x,1/x)}+1728-e^{2\pi}.
$$
The lemma follows immediately upon taking logarithms.
\end{proof}

To profitably apply Lemma~\ref{lemma:jbound} to (\ref{Sl1}), we first shift the argument of $j((it+b)/l)$ 
using the invariance of $j(z)$ under $z \mapsto z+1$. Fixing the integer $N=[\sqrt{l/t}]$, let $I_N$ 
be the unit interval $[1/(N+1),(N+2)/(N+1)]$.  We may rewrite (\ref{Sl1}) as
\begin{equation}\label{Sl2}
S(l,t) = \sum_{b/l\in I_N}\bigl\|j((it+b)/l)\bigr\|,
\end{equation}
where $b$ now ranges over the $l$ integers in the interval $[l/(N+1),l(N+2)/(N+1))$.

Using the Farey series of order $N$, we may partition $I_N$ into subintervals as
\begin{equation}\label{Ipartition}
I_N=\bigcup_{k=1}^{N}\bigcup_{\substack{h=1\\(h,k)=1}}^k I_N(h/k),
\end{equation}
where $I_N(h/k)=[\rho_1,\rho_2)$ contains $h/k$.  The endpoints $\rho_1$ and $\rho_2$ both satisfy
\begin{equation}\label{endpoints}
1/(2kN)\le |\rho_i-h/k|\le 1/(k(N+1)),
\end{equation}
see \cite[Ch.~3]{Hardy:NumberTheory} and \cite[Lemma 3]{CohenPaula:ModularPolynomials}.
There are $K_N=\sum_{k=1}^N\phi(k)$ subintervals in (\ref{Ipartition}), and each contains at least 
one $b/l\in I_N$.
The bound $K_N=3N^2/\pi^2+O(N\log N)$ is well known \cite[Thm.~3.7]{Apostol:NumberTheory}, and for $N\ge 2$ we have the explicit bound
$$
K_N=\sum_{k=1}^N\phi(k) \le \frac{3}{\pi^2}N^2+\frac{1}{2}N\log N,
$$
proven in Lemma~\ref{lemma:A2} of the appendix (see \cite{Walfisz:PhiSumBound} for an asymptotically tighter 
bound).

We now assume $l > 5$, so that $N\ge 2$.  The bound we eventually prove will also hold when $l\le 5$.
For each $I_N(h/k)$ in $(\ref{Ipartition})$, let us pick integers $r$ and $s$ with $rk-sh=1$.
We may then apply the unimodular transformation
$$
\Lambda = \left(
\begin{matrix}
s & -r\\
k & -h\\
\end{matrix} \right)
$$
to $z=(it+b)/l$ in each term of (\ref{Sl2}).  
Using $\Im \Lambda z = \Im z/\bigl|kz-h\bigr|^2$, we obtain
$$
\Im \Lambda z = \frac{t/l}{k^2\bigl|b/l-h/k\bigr|^2+k^2(t/l)^2}.
$$
Now let $\alpha=k\sqrt{l/t}\medspace\bigl|b/l-h/k\bigr| \le 1$ and 
$\beta = k\sqrt{t/l}\le 1$, so that
\begin{equation*}
\Im \Lambda z \medspace =\medspace \frac{1}{\alpha^2+\beta^2}\medspace \ge \medspace \frac{1}{2}.
\end{equation*}
Note that each $b/l\in I_N$ is contained in a particular $I_N(h/k)$ that determines $\alpha$ and $\beta$.
We may now bound $S(l,t)$ by
\begin{equation}\label{Sl3}
S(l,t) \medspace \le \medspace \sum_{k=1}^N\sum_{\substack{h=1\\(k,h)=1}}^k \sum_{b/l\in I_N(h/k)} \bigl(g_{h/k}(b/l) + \varepsilon_{h/k}(b/l)\bigr).
\end{equation}
The notation in the right-hand side of~(\ref{Sl3}) is as follows.  For each $I_N(h/k)$ we pick a 
unimodular matrix $\Lambda$ as above, and with $z = (it+b)/l$ we put 
$g_{h/k}(b/l) = 2\pi \Im\Lambda z$ and $\varepsilon_{h/k}(b/l) = c(\Im \Lambda z)$, where $c(x)$ is the function defined in Lemma~\ref{lemma:jbound}.

As the following lemma shows, the na\"\i ve bound $\varepsilon_{h/k}(b/l) \le c(1/2)
< 9.429$ that we get from Lemma~\ref{lemma:jbound} can be significantly improved, on average.

\begin{lemma}\label{lemma:epsilonbound}
With the notation as above, let $\varepsilon(l,t) = \sum_{h,k,b} 
\varepsilon_{h/k}(b/l)$. For all $l > 5$ and $t\in[1,1.254)$ we have
$$
\varepsilon(l,t) \sle (3.066+1-t)l + 2.485\sqrt{l}\log l + 36.963.
$$
\end{lemma}

\begin{proof}

Among the points of the form $b/l$ in the interval $I_N(h/k)$, we call the smallest and largest
\emph{exterior} points.  The others are called \emph{interior} points.  For each interior point
$b/l\in I_N(h/k)$ we have
$$
|b/l-h/k| \sle 1/(k(N+1)) - 1/l \slt 1/(k\sqrt{l/t})- 1/l \seq t/(\beta l)-1/l,
$$
yielding $\alpha = (\beta l/t)|b/l-h/k| < 1-\beta/t$.

We now divide the exterior points into two categories.  We call an exterior point \emph{good} if it 
is at least a distance $1/(2l)$ from both endpoints of the interval in which it lies. An exterior point
is \emph{bad} otherwise. By a calculation similar to that above, we find that 
$\alpha < 1-\beta/(2t)$ holds for each good exterior point.

There are $2K_N$ half intervals of the form $[\rho_0,h/k)$ or $[h/k,\rho_1)$, and each contains at most 
one exterior point. Adjacent half intervals that do not lie in the same $I_N(h/k)$ cannot both contain a 
bad exterior point.  Hence, there are at most $K_N+1$ bad exterior points.  For the purposes of 
computing an upper bound, we assume that there are exactly this many, and that every interval has two 
exterior points, implying that there are exactly $K_N-1$ good exterior points.

In the worst case, the bad exterior points lie in the intervals $I_N(h/k)$ with the largest values 
of~$\beta$.  If we order the intervals according to $\beta\approx k/N$ and divide them into quartiles, 
this roughly means that each interval in the top two quartiles contains two bad exterior points, while 
each interval in the bottom two quartiles contains two good exterior points.  By Corollary~\ref{cor:A3} of 
the appendix, up to an absolute error of 2, at least \nicefrac{1}{4} of the intervals have $\beta \le C_1= 0.539$,
at least \nicefrac{1}{2} have $\beta\le C_2=0.742$, and at least $\nicefrac{3}{4}$ have $\beta \le C_3 = 0.917$.

Let $f(\alpha,\beta)=c(1/(\alpha^2+\beta^2))$, with $c(x)$ as in Lemma~\ref{lemma:jbound}.
We use $c_1=f(1,1)$ and $c_2=f(1,C_3)$ to estimate $\varepsilon_{h/k}(b/l)$ for the bad exterior points in the upper 
two quartiles, and we use $c_3=f(1-C_2/(2t),C_2)$ and $c_4=f(1-C_1/(2t),C_1)$ for the good 
exterior points in the lower two quartiles.

For the interior points, we note that for $\beta > C_3 > 2t/3$, the interval $I_N(h/k)$ has width 
less than $3/l$ and can contain at most one interior point.  Similarly, if we have $\beta > C_2 > t/2$, then
there can be at most two interior points.  We may therefore use $c_5=f(1-1/t,1)$ to estimate $\varepsilon_{h/k}(b/l)$ 
for the interior points in the top quartile (at most one per interval), $c_6=f(1-C_3/t,C_3)$ 
for the interior points in the second quartile (at most two per interval), and $c_7=f(1-C_2/t,C_2)$ for the remaining interior points.
Applying Lemma~\ref{lemma:A2}, we use $M=3l/(2\pi^2t) + \frac{1}{8}\sqrt{l/t}\log(l/t)$ to bound $K_N/2$.
Putting everything together, we obtain the function
\begin{align*}
\varepsilon'(l,t) &\seq c_1(M+4) + c_2M + c_3M + c_4(M-4)\\ &\hspace{62pt} + 0.5c_5(M+2) + c_6M + c_7(l-5.5M-2)\\
 &\seq (c_1+c_2+c_3+c_4+0.5c_5+c_6-5.5c_7)M + 4(c_1-c_4) + (c_5-2c_7) + c_7l,
\end{align*}
as an upper bound on $\varepsilon(l,t)$.
For any fixed $l > 5$ we find that for $t\in[1,1.254)$ the function $\varepsilon'(l,t)$ is maximized when $t=1$, and in fact the bound
$$
\varepsilon(l,t)\sle \varepsilon'(l,t)\sle \varepsilon'(l,1)+(1-t)l
$$
holds for all $l > 5$ and $t\in[1,1.254)$.  Computing $\varepsilon'(l,1)$ then yields the lemma.
\end{proof}

Our remaining task is to bound the terms $g_{h/k}(b/l)$ appearing in~(\ref{Sl3}).
Setting $\theta=lh/k$, we have
$$
g_{h/k}(b/l) = 2\pi k^{-2}tl^{-1}/(t^2l^{-2} + (b/l-h/k)^2) = 2\pi lt k^{-2}\bigl(t^2+(b-\theta)^2\bigr)^{-1},
$$
which we view as a function of $b\in\Z$.  We now apply the identity
\begin{equation}\label{expseries}
\sum_{c=-\infty}^{+\infty}\bigl(t^2+(c-\theta)^2\bigr)^{-1} = \pi t^{-1}\sum_{\nu=\-\infty}^{+\infty} e^{-2\pi|\nu|t}e^{2\pi i \nu \theta}\qquad (t > 0),
\end{equation}
as in \cite[Lemma 6]{CohenPaula:ModularPolynomials}, to obtain
\begin{equation}\label{ghksum}
\sum_{b/l\in I_N(h/k)} g_{h/k}(b/l)\medspace \le\medspace \sum_{c=-\infty}^{+\infty}g_{h/k}(c/l)\medspace =\medspace 2\pi^2lk^{-2}\sum_{\nu=-\infty}^{+\infty}e^{-2\pi|\nu|t}e^{2\pi i \nu lh/k},
\end{equation}
and we use the following lemma to bound the right-hand side.

\begin{lemma}\label{lemma:tsumbound}
Let $l$ be prime, let $t\in\R_{\ge 1}$, and for $\nu\in\Z$, define $a_\nu=2\pi^2 l e^{-2\pi|\nu|t}$.
Assume $N=[\sqrt{l/t}] \ge 2$.  Then we have
$$
\sum_{k=1}^N \sum_{\substack{h=1\\(k,h)=1}}^k  \sum_{\nu=-\infty}^{+\infty} k^{-2} a_\nu e^{2\pi i \nu lh/k}\slt
6l\log l \splus (13.889-6\log t)l \splus 3.290\sqrt{l}\log l\splus 6.580\sqrt{l}.$$
\end{lemma}
\begin{proof}
Noting absolute convergence, we reorder the sums, obtaining
\begin{equation}\label{triplesum}
\sum_{k=1}^N \sum_{\substack{h=1\\(k,h)=1}}^k  \sum_{\nu=-\infty}^{+\infty} k^{-2} a_\nu e^{2\pi i \nu lh/k} \medspace = \medspace
\sum_{\nu=-\infty}^{+\infty}  a_\nu  \sum_{k=1}^N k^{-2} \sum_{\substack{h=1\\(k,h)=1}}^k e^{2\pi i \nu lh/k}.
\end{equation}
With $n=\nu l$ the inner sum becomes
$$
c_k(n)=\sum_{\substack{h=1\\(h,k)=1}}^k e^{2\pi inh/k}.
$$
For $n>0$ this is a Ramanujan sum, and we may apply the identity
$$
c_k(n) = \frac{\mu\bigl(k/(k,n)\bigr)\phi(k)}{\phi\bigl(k/(k,n)\bigr)},
$$
due to H\"older \cite{Holder:RamanujanSum} (or see \cite[Thm.~2]{Anderson:RamanujanSum}).
We then have $|c_k(n)|\le (k,n)$ for all $n>0$, and for $n<0$ we note that $c_k(-n)=\overline{c_k(n)}=c_k(n)$.
For $\nu\ne 0$ we therefore have
$$
\sum_{k=1}^N k^{-2} c_k(\nu l) \le \sum_{k=1}^{\infty} k^{-2}(k,\nu)\le\sum_{d|\nu}d\sum_{n=1}^{\infty}(nd)^{-2}= \zeta(2)\sum_{d|\nu}d^{-1}=\zeta(2)\sigma(|\nu|)/|\nu|,
$$
where $\zeta(2)=\sum_{n=1}^{\infty}n^{-2}=\pi^2/6$.  Since $c_k(0)=\phi(k)$, we may bound (\ref{triplesum}) by
\begin{equation}\label{twosums}
a_0\sum_{k=1}^N \phi(k)/k^{2} \medspace +\medspace \frac{\pi^2}{3}\sum_{\nu=1}^{\infty}a_{\nu}\sigma(\nu)/\nu.
\end{equation}
To treat the infinite sum in (\ref{twosums}) we evaluate the first two terms and bound the tail with an integral, using $\sigma(\nu)=\sum_{d|\nu}d\le \nu^2$.  This yields
\begin{equation}\label{rightsum}
\sum_{\nu=1}^{\infty}a_{\nu}\sigma(\nu)/\nu\medspace\le\medspace 2\pi^2le^{-2\pi t} + 3\pi^2le^{-4\pi t} + 2\pi^2l \sum_{\nu=3}^\infty e^{-2\pi \nu t}\nu\medspace<\medspace 0.037\cdot l,
\end{equation}
where we have used
$$
\sum_{\nu=3}^\infty e^{-2\pi \nu t}\nu \le \int_2^{\infty} e^{-2\pi x}x dx = \frac{e^{-4 \pi} (1+4\pi)}{4 \pi^2} < 1.2\times 10^{-6},
$$
valid for $t\ge 1$.  We now apply Lemma~\ref{lemma:A1} of the appendix, whose error term is an increasing function of $x=N$, to the first sum in (\ref{twosums}).  Using $N=[\sqrt{l/t}]\le \sqrt{l}$ we obtain
\begin{align}\label{leftsum}
a_0\sum_{k\le\sqrt{l}}\phi(k)/k^{2}\medspace &\sle 2\pi^2l\left(\frac{3\log l}{\pi^2}+\frac{6\gamma}{\pi^2}-C+ \frac{\log l}{6\sqrt{l}}+\frac{1}{3\sqrt{l}}\right)\\\notag
                                       &\sle 6l\log l \splus (13.767-6\log t)l \splus 3.290\sqrt{l}\log l \splus 6.580\sqrt{l}.
\end{align}
Applying (\ref{leftsum}) and (\ref{rightsum}) to (\ref{twosums}) yields the lemma.
\end{proof}

\begin{remark}
Extending the sum in (\ref{ghksum}) to all integers allowed us to conveniently bound 
the triple sum in (\ref{Sl3}) via Lemma~\ref{lemma:tsumbound}. This yields the correct leading term of $6l\log l$ in Theorem~1, 
however the $O(l)$ term is overestimated significantly.  Lemma~\ref{lemma:ghktail} proves a lower bound on the tail of 
the middle sum in (\ref{ghksum}) which is then subtracted from the bound in Lemma~\ref{lemma:tsumbound} to
sharpen the $O(l)$ term.
\end{remark}

\begin{lemma}\label{lemma:ghktail} Let the notation be as above.  For $l> 3600$ we have
$$
\sum_{b/l\in I_N(h/k)} g_{h/k}(b/l) \leq 
6l\log l \splus (3.803-6\log t)l \splus 17.693\sqrt{l}\log l\ - 58.939\sqrt{l}.
$$
\end{lemma}

\begin{proof}
Consider an interval $I_N(h/k)$ for which $k>\sqrt{lt}/2$.  By (\ref{endpoints}), the width of 
$I_N(h/k)$ is less than $2/l$, hence it contains at most two points of the form $b/l$.  It follows 
that at most two of the terms $g_{h/k}(c/l)$ in the middle sum of (\ref{ghksum}) correspond to terms 
present in the left sum.  We assume there are exactly two, and these must be of the form $g_{h/k}(b/l)$ and $g_{h/k}((b+1)/l)$.
The sum is maximized when $\theta=b+\frac{1}{2}$, hence the
overestimate introduced by the inequality in (\ref{ghksum}) is at least
\begin{equation}\label{tanhbound}
2\pi lk^{-2} \sum_{n\ne0,1}\frac{t}{t^2+(n-\frac{1}{2})^2} \seq 2\pi lk^{-2}\left(\pi\tanh \pi t - \frac{2t}{t^2+(\frac{1}{2})^2}\right),
\end{equation}
where we have used the identity $\pi\tanh \pi t = \sum_{n=-\infty}^\infty \frac{t}{t^2+(n-\frac{1}{2})^2}$.

For each $k\in(\sqrt{lt}/2,\sqrt{l/t}]$, there are $\phi(k)$ intervals $I_N(h/k)$ to 
which (\ref{tanhbound}) applies.  The total overestimate related to these intervals is at 
least
\begin{equation}\label{overhk}
2\pi l\left(\pi\tanh \pi t - \frac{2t}{t^2+(\frac{1}{2})^2}\right) \sum_{\sqrt{lt}/2<k\le\sqrt{l/t}}\frac{\phi(k)}{k^2}.
\end{equation}
To bound the sum in (\ref{overhk}), we apply Lemma~\ref{lemma:A1} twice and take the difference:
\begin{equation}\label{overhksum}
\sum_{\sqrt{lt}/2<k\le\sqrt{l/t}}\frac{\phi(k)}{k^2} \sge \frac{6}{\pi^2}\log\frac{2}{t} \splus \epsilon_1(l,t).
\end{equation}
The term $\epsilon_1(l,t)$ arises from the error term $\epsilon\bigl((\log x +1)/(3x)\bigr)$ in Lemma~\ref{lemma:A1}.
We make the worst-case assumption that the upper error is maximally negative and the lower error is maximally 
positive, yielding
\begin{equation}\label{eps1}
\epsilon_1(l,t) = -\medspace \frac{(t+2)\log l+2t+4-\log 16 +(2-t)\log t}{6\sqrt{lt}}.
\end{equation}
Combining (\ref{overhk}) and (\ref{overhksum}) we obtain the lower bound
$$
\delta_1(l,t) = 2\pi l\left(\pi\tanh \pi t - \frac{2t}{t^2+(\frac{1}{2})^2}\right)\left(\frac{6}{\pi^2}\log\frac{2}{t} \splus \epsilon_1(l,t)\right).
$$

We now apply the same argument to intervals $I_N(h/k)$ with $\sqrt{lt}/3<k\le \sqrt{lt}/2$. For such 
intervals there are at most three terms in the middle sum of (\ref{ghksum}) that correspond to terms in 
the left sum, and the sum of these terms is maximized when the middle term has $\theta=b$.  Using the 
identity $\pi\coth \pi t =\sum_{n=-\infty}^\infty \frac{t}{t^2+n^2}$, we obtain
$$
\delta_2(l,t) \seq 2\pi l \left(\pi\coth \pi t - \frac{1}{t} - \frac{2t}{t^2+1}\right)\left(\frac{6}{\pi^2}\log\frac{3}{2} \splus \epsilon_2(l,t)\right).
$$
To compute $\epsilon_2(l,t)$, we now take the difference of the maximal 
positive errors arising from Lemma~\ref{lemma:A1}, consistent with our worst-case assumption 
in (\ref{eps1}) above, obtaining
$$
\epsilon_2(l,t) = -\medspace \frac{\log lt - 2(3\log 3-2\log 2) + 2}{6\sqrt{lt}}.
$$

Continuing in this fashion, we consider $I_N(h/k)$ with $\sqrt{lt}/d < k \le \sqrt{lt}/(d+1)$ for an 
integer $d>2$.  Let us define
$$
S_d \seq \frac{6}{\pi^2}\log\left(\frac{d+1}{d}\right)\sminus\frac{\log lt - 2\bigl((d+1)\log(d+1)-d\log d\bigl) + 2}{6\sqrt{lt}}.
$$
For $l\ge (d+1)^2$ and $d > 2$ we then obtain the lower bound
\begin{align}
\delta_d \ge
\begin{cases}
2\pi lS_d\left(\pi\tanh\pi - \sum_{|n-\frac{1}{2}|\le\frac{d}{2}}\frac{t}{t^2+(n-\frac{1}{2})^2}\right)\qquad&\text{if $d$ is odd;}\\
\\
2\pi lS_d\left(\pi\coth\pi - \sum_{|n|\le \frac{d}{2}}\frac{t}{t^2+n^2}\right)\qquad&\text{if $d$ is even.}\\
\end{cases}
\end{align}
Assuming $l\ge 3600$, we let $\delta(l,t) = \delta_1(l,t) + \cdots + \delta_{59}(l,t)$.
For any fixed $l\ge 3600$, as $t$ varies over $[1,1.254)$ we find that $\delta(l,t)$ is minimized when $t=1$ (we note that this is not true of the first few $\delta_d(l,t)$).  We then compute
\begin{equation}\label{deltabound}
\delta(l,1) = \delta_1(l,1) + \cdots + \delta_{59}(l,1) \sge 10.086\cdot l \sminus 17.693\sqrt{l}\log l  \splus 58.939\sqrt{l}
\end{equation}
as a lower bound on the total overestimate in the middle sum of (\ref{ghksum}).
Subtracting the right-hand side of (\ref{deltabound}) from the bound in Lemma~6 yields the lemma.
\end{proof}

\begin{thm1}
For every prime $l$ we have $h(\Phi_l)\sle 6l\log l \splus 16l \splus 14\sqrt{l}\log l$.
\end{thm1}

\begin{proof}
The modular polynomials $\Phi_l$ are well known for $l=2,3,5$, and the theorem holds in these cases,
so assume $l>5$.
Let $y\in [1728,3456]$, and let $t\in[1,1.254)$ satisfy $j(it)=y$, as discussed following Lemma~\ref{lemma:interpolation}.

From Lemma~\ref{lemma:hlybound} we have
\begin{equation}\label{hly}
h(\Phi_{l,y}) \le \bigl\|j(ilt)\bigr\| + S(l,t) + (\log 2)l,
\end{equation}
where $S(l,t)$ is the sum in (\ref{Sl3}).  Applying Lemma~\ref{lemma:jbound}, we obtain
\begin{equation}\label{S1}
\bigl\|j(ilt)\bigr\| \slt 2\pi lt + 1.172.
\end{equation}
For $S(l,t)$ we have two different bounds, depending on whether $l$ is less than or greater
than 3600, and this yields two bounds for $h(\Phi_l)$.

For $5 < l < 3600$, we bound $S(l,t)$ by adding the bound in Lemma~\ref{lemma:epsilonbound} to the bound in Lemma~\ref{lemma:tsumbound}.  Plugging this and (\ref{S1}) into (\ref{hly}) yields
\begin{equation}\label{hlyt1}
h(\Phi_{l,y})\sle 6l\log l + (18.649 + 2\pi t -t-6\log t)l + 5.775\cdot\sqrt{l}\log l + 6.580\sqrt{l} + 36.963.
\end{equation}
For $t\in[1,1.254)$ this bound is maximized when $t=1$.  Lemma~\ref{lemma:interpolation} then yields
\begin{equation}\label{B1}
h(\Phi_{l})\sle B_1(l) = 6l\log l + 26.016\cdot l + 5.775\cdot\sqrt{l}\log l + 6.580\sqrt{l} + 39.046.
\end{equation}

For $l \ge 3600$ we instead bound $S(l,t)$ by adding the bound in Lemma~\ref{lemma:epsilonbound} to the bound in Lemma~\ref{lemma:ghktail}.
In this case we obtain
\begin{equation}\label{B2}
h(\Phi_l)\sle B_2(l) \seq 6l\log l + 15.929\cdot l + 20.178 \sqrt{l}\log l - 58.939 \sqrt{l} + 39.046.
\end{equation}

A direct computation, using the bound $B_1(l)$ and the algorithm described 
in~\cite{BrokerLauterSutherland:CRTModPoly}, finds that Theorem~1 holds for $l < 3600$ (see Table 1).
For all $l \ge 3600$, the bound of the theorem is greater than $B_2(l)$, hence the Theorem~1 holds for $l \ge 3600$.
\end{proof}

\begin{corollary}\label{cor1}
For every prime $l$ we have $h(\Phi_l)\sle 6l\log l \splus 18l$.
\end{corollary}
\begin{proof}
For $l \ge 3600$ the corollary follows from the bound $B_2(l)$ in (\ref{B2}).
For $l < 3600$ it is verified by the values of $h(\Phi_l)$ in Table~1.
\end{proof}

Table~1 is listed on the following three pages in a three column format.  For each prime $l\le 3607$, the table lists three values.  The first is the integer
$$h_2(\Phi_l)=\min \{ n : 2^n > h(\Phi_l)\}.$$
This represents the number of bits required to store the absolute value of the largest coefficient of $\Phi_l$, and we have $h(\Phi_l)\approx h_2(\Phi_l)\log 2$.  The second is
$$c_l = (h(\Phi_l) - 6l\log l) / l,$$
which reflects the constant in the $O(l)$ term of $h(\Phi_l)$.  The third value is
$$r_l = (6l\log l+ 18l) / h(\Phi_l),$$
the ratio of the bound in Corollary~\ref{cor1} to the height of $\Phi_l$.
\medskip

Within the range of Table~1, we find that for $l> 157$ the bound of Corollary~\ref{cor1} exceeds $h(\Phi_l)$ by less than fifteen percent, as indicated by the values for $r_l$.
However, as shown by the values for $c_l$, our bounds are clearly not optimal.  Based on the data in the table, we make the following conjectures.

\begin{conjecture}
For every prime $l > 30$ we have $h(\Phi_l)\slt 6l\log l + 12l$.
\end{conjecture}
\begin{conjecture}
For $l$ prime we have $\liminf_{l\to\infty} (h(\Phi_l)-6l\log l) / l > 11.8$.
\end{conjecture}

\begin{table}
\begin{center}
\begin{footnotesize}
\begin{tabular}{@{}rrrr|rrrr|rrrr@{}}
$l$ &$h_2(\Phi_l)$&$c_l$&$r_l$&$l$ &$h_2(\Phi_l)$&$c_l$&$r_l$&$l$ &$h_2(\Phi_l)$&$c_l$&$r_l$\\
\midrule
2&48&12.48&1.33&269&17634&11.87&1.13&617&44862&11.85&1.12\\
3&71&9.81&1.50&271&17767&11.83&1.14&619&45025&11.85&1.12\\
5&157&12.11&1.27&277&18207&11.82&1.14&631&45980&11.82&1.12\\
7&220&10.11&1.36&281&18513&11.84&1.13&641&46794&11.82&1.12\\
11&421&12.14&1.22&283&18673&11.86&1.13&643&46980&11.85&1.12\\
13&496&11.06&1.26&293&19438&11.90&1.13&647&47306&11.85&1.12\\
17&705&11.75&1.22&307&20483&11.89&1.13&653&47798&11.85&1.12\\
19&796&11.37&1.23&311&20799&11.92&1.13&659&48280&11.84&1.12\\
23&1025&12.08&1.19&313&20941&11.90&1.13&661&48429&11.82&1.12\\
29&1348&12.02&1.19&317&21260&11.93&1.13&673&49423&11.83&1.12\\
31&1440&11.59&1.20&331&22304&11.89&1.13&677&49761&11.84&1.12\\
37&1767&11.44&1.20&337&22767&11.91&1.13&683&50242&11.83&1.12\\
41&2012&11.73&1.18&347&23542&11.93&1.13&691&50915&11.84&1.12\\
43&2122&11.64&1.19&349&23676&11.89&1.13&701&51731&11.84&1.12\\
47&2376&11.94&1.17&353&23990&11.91&1.13&709&52368&11.81&1.12\\
53&2739&12.00&1.17&359&24450&11.91&1.13&719&53237&11.86&1.12\\
59&3104&12.00&1.16&367&25053&11.89&1.13&727&53905&11.86&1.12\\
61&3213&11.84&1.17&373&25511&11.88&1.13&733&54397&11.86&1.12\\
67&3581&11.82&1.17&379&25962&11.86&1.13&739&54884&11.85&1.12\\
71&3841&11.92&1.16&383&26285&11.88&1.13&743&55230&11.86&1.12\\
73&3952&11.78&1.17&389&26730&11.85&1.13&751&55891&11.86&1.12\\
79&4328&11.76&1.16&397&27364&11.87&1.13&757&56380&11.85&1.12\\
83&4591&11.83&1.16&401&27642&11.82&1.13&761&56722&11.86&1.12\\
89&4968&11.76&1.16&409&28259&11.81&1.13&769&57373&11.84&1.12\\
97&5476&11.68&1.16&419&29057&11.84&1.13&773&57753&11.89&1.12\\
101&5751&11.78&1.16&421&29188&11.80&1.13&787&58908&11.87&1.12\\
103&5881&11.77&1.16&431&29993&11.84&1.13&797&59763&11.89&1.12\\
107&6158&11.85&1.15&433&30128&11.80&1.13&809&60743&11.87&1.12\\
109&6282&11.80&1.16&439&30616&11.83&1.13&811&60917&11.88&1.12\\
113&6561&11.88&1.15&443&30916&11.81&1.13&821&61764&11.88&1.12\\
127&7500&11.87&1.15&449&31385&11.81&1.13&823&61913&11.87&1.12\\
131&7783&11.93&1.15&457&32001&11.79&1.13&827&62282&11.89&1.12\\
137&8183&11.88&1.15&461&32351&11.84&1.13&829&62429&11.88&1.12\\
139&8318&11.87&1.15&463&32491&11.82&1.13&839&63267&11.88&1.12\\
149&9009&11.89&1.15&467&32818&11.83&1.13&853&64459&11.89&1.12\\
151&9145&11.88&1.15&479&33779&11.85&1.13&857&64774&11.87&1.12\\
157&9561&11.87&1.15&487&34410&11.85&1.13&859&64973&11.89&1.12\\
163&9979&11.87&1.14&491&34730&11.85&1.13&863&65300&11.89&1.12\\
167&10272&11.93&1.14&499&35353&11.83&1.13&877&66485&11.89&1.12\\
173&10690&11.91&1.14&503&35686&11.85&1.13&881&66805&11.87&1.12\\
179&11111&11.90&1.14&509&36165&11.85&1.12&883&66970&11.87&1.12\\
181&11238&11.85&1.14&521&37127&11.86&1.12&887&67354&11.91&1.12\\
191&11951&11.86&1.14&523&37280&11.85&1.12&907&69017&11.88&1.12\\
193&12089&11.84&1.14&541&38717&11.84&1.12&911&69340&11.87&1.12\\
197&12369&11.82&1.14&547&39203&11.85&1.12&919&70036&11.88&1.12\\
199&12509&11.81&1.14&557&40015&11.86&1.12&929&70884&11.88&1.12\\
211&13359&11.77&1.14&563&40495&11.86&1.12&937&71530&11.86&1.12\\
223&14229&11.78&1.14&569&40949&11.82&1.12&941&71915&11.89&1.12\\
227&14524&11.80&1.14&571&41142&11.86&1.12&947&72396&11.87&1.12\\
229&14651&11.74&1.14&577&41601&11.83&1.12&953&72916&11.88&1.12\\
233&14949&11.77&1.14&587&42408&11.83&1.12&967&74105&11.87&1.12\\
239&15403&11.81&1.14&593&42918&11.85&1.12&971&74434&11.86&1.12\\
241&15525&11.74&1.14&599&43404&11.85&1.12&977&74951&11.87&1.12\\
251&16281&11.81&1.14&601&43542&11.83&1.12&983&75471&11.87&1.12\\
257&16721&11.80&1.14&607&44041&11.84&1.12&991&76128&11.85&1.12\\
263&17176&11.84&1.14&613&44523&11.83&1.12&997&76672&11.88&1.11\\
\midrule
\end{tabular}
\end{footnotesize}
\end{center}
\end{table}

\begin{table}
\begin{center}
\begin{footnotesize}
\begin{tabular}{@{}rrrr|rrrr|rrrr@{}}
$l$ &$h_2(\Phi_l)$&$c_l$&$r_l$&$l$ &$h_2(\Phi_l)$&$c_l$&$r_l$&$l$ &$h_2(\Phi_l)$&$c_l$&$r_l$\\
\midrule

1009&77653&11.84&1.12&1427&114111&11.85&1.11&1823&149632&11.84&1.11\\
1013&78047&11.88&1.11&1429&114239&11.82&1.11&1831&150434&11.87&1.11\\
1019&78557&11.88&1.11&1433&114631&11.84&1.11&1847&151842&11.86&1.11\\
1021&78700&11.86&1.11&1439&115174&11.85&1.11&1861&153094&11.85&1.11\\
1031&79548&11.85&1.11&1447&115888&11.85&1.11&1867&153712&11.87&1.11\\
1033&79708&11.84&1.12&1451&116239&11.85&1.11&1871&154036&11.86&1.11\\
1039&80247&11.86&1.11&1481&118903&11.85&1.11&1873&154203&11.85&1.11\\
1049&81105&11.86&1.11&1453&116374&11.83&1.11&1877&154618&11.87&1.11\\
1051&81295&11.87&1.11&1483&119076&11.84&1.11&1879&154792&11.87&1.11\\
1061&82129&11.85&1.11&1459&116933&11.84&1.11&1889&155659&11.85&1.11\\
1063&82311&11.86&1.11&1487&119416&11.84&1.11&1901&156763&11.86&1.11\\
1069&82794&11.84&1.11&1489&119609&11.84&1.11&1907&157305&11.86&1.11\\
1087&84360&11.85&1.11&1471&118003&11.84&1.11&1913&157904&11.88&1.11\\
1091&84674&11.83&1.11&1493&119986&11.85&1.11&1931&159541&11.87&1.11\\
1093&84859&11.83&1.11&1499&120515&11.85&1.11&1933&159666&11.85&1.11\\
1097&85237&11.86&1.11&1511&121585&11.85&1.11&1949&161130&11.85&1.11\\
1103&85739&11.85&1.11&1523&122646&11.85&1.11&1951&161341&11.86&1.11\\
1109&86259&11.85&1.11&1531&123383&11.86&1.11&1973&163364&11.87&1.11\\
1117&86961&11.85&1.11&1543&124421&11.84&1.11&1979&163887&11.86&1.11\\
1123&87456&11.84&1.11&1549&124968&11.85&1.11&1987&164625&11.86&1.11\\
1129&87968&11.83&1.11&1553&125350&11.86&1.11&1993&165166&11.86&1.11\\
1151&89905&11.85&1.11&1559&125853&11.84&1.11&1997&165566&11.87&1.11\\
1153&90047&11.83&1.11&1567&126622&11.87&1.11&1999&165699&11.85&1.11\\
1163&90905&11.83&1.11&1571&126963&11.86&1.11&2003&166095&11.86&1.11\\
1171&91633&11.85&1.11&1579&127692&11.87&1.11&2011&166855&11.87&1.11\\
1181&92502&11.85&1.11&1583&128008&11.85&1.11&2017&167390&11.87&1.11\\
1187&93011&11.84&1.11&1597&129282&11.86&1.11&2027&168251&11.85&1.11\\
1193&93538&11.84&1.11&1601&129617&11.85&1.11&2029&168460&11.86&1.11\\
1201&94221&11.83&1.11&1607&130196&11.86&1.11&2039&169405&11.87&1.11\\
1213&95278&11.84&1.11&1609&130333&11.85&1.11&2053&170690&11.87&1.11\\
1217&95657&11.86&1.11&1613&130752&11.87&1.11&2063&171582&11.86&1.11\\
1223&96143&11.84&1.11&1619&131249&11.85&1.11&2069&172121&11.85&1.11\\
1229&96679&11.84&1.11&1621&131399&11.84&1.11&2081&173241&11.86&1.11\\
1231&96833&11.83&1.11&1627&131992&11.87&1.11&2083&173405&11.85&1.11\\
1237&97370&11.84&1.11&1637&132857&11.85&1.11&2087&173815&11.87&1.11\\
1249&98397&11.83&1.11&1657&134687&11.86&1.11&2089&173958&11.85&1.11\\
1259&99286&11.83&1.11&1663&135204&11.86&1.11&2099&174910&11.86&1.11\\
1277&100871&11.84&1.11&1667&135563&11.86&1.11&2111&176016&11.87&1.11\\
1279&101041&11.84&1.11&1669&135747&11.86&1.11&2113&176184&11.86&1.11\\
1283&101402&11.84&1.11&1693&137935&11.87&1.11&2129&177660&11.86&1.11\\
1289&101890&11.82&1.11&1697&138249&11.85&1.11&2131&177843&11.86&1.11\\
1291&102103&11.84&1.11&1699&138458&11.86&1.11&2137&178382&11.86&1.11\\
1297&102577&11.81&1.11&1709&139348&11.86&1.11&2141&178787&11.87&1.11\\
1301&102968&11.83&1.11&1733&141513&11.86&1.11&2143&178916&11.85&1.11\\
1303&103114&11.82&1.11&1721&140437&11.86&1.11&2153&179920&11.88&1.11\\
1307&103491&11.83&1.11&1723&140607&11.85&1.11&2161&180570&11.85&1.11\\
1319&104571&11.85&1.11&1741&142238&11.86&1.11&2179&182243&11.85&1.11\\
1321&104705&11.82&1.11&1747&142829&11.88&1.11&2203&184512&11.87&1.11\\
1327&105265&11.84&1.11&1753&143314&11.85&1.11&2207&184868&11.86&1.11\\
1361&108239&11.83&1.11&1759&143851&11.85&1.11&2213&185358&11.84&1.11\\
1367&108774&11.83&1.11&1777&145498&11.86&1.11&2221&186162&11.86&1.11\\
1373&109286&11.82&1.11&1783&146025&11.85&1.11&2237&187665&11.87&1.11\\
1381&109991&11.82&1.11&1787&146440&11.87&1.11&2239&187831&11.87&1.11\\
1399&111619&11.84&1.11&1789&146604&11.87&1.11&2243&188184&11.86&1.11\\
1409&112472&11.83&1.11&1801&147642&11.85&1.11&2251&188957&11.87&1.11\\
1423&113740&11.84&1.11&1811&148577&11.86&1.11&2267&190415&11.86&1.11\\
\midrule
\end{tabular}
\end{footnotesize}
\end{center}
\end{table}

\begin{table}
\begin{center}
\begin{footnotesize}
\begin{tabular}{@{}rrrr|rrrr|rrrr@{}}
$l$ &$h_2(\Phi_l)$&$c_l$&$r_l$&$l$ &$h_2(\Phi_l)$&$c_l$&$r_l$&$l$ &$h_2(\Phi_l)$&$c_l$&$r_l$\\
\midrule

2269&190584&11.86&1.11&2699&230697&11.84&1.10&3169&275250&11.84&1.10\\
2273&190989&11.87&1.11&2707&231444&11.84&1.10&3181&276425&11.84&1.10\\
2281&191712&11.86&1.11&2711&231884&11.86&1.10&3187&277030&11.85&1.10\\
2287&192310&11.88&1.11&2713&232018&11.84&1.10&3191&277388&11.85&1.10\\
2293&192816&11.86&1.11&2719&232569&11.84&1.10&3203&278548&11.85&1.10\\
2297&193195&11.86&1.11&2729&233574&11.86&1.10&3209&279143&11.85&1.10\\
2309&194314&11.86&1.11&2731&233729&11.85&1.10&3217&279834&11.84&1.10\\
2311&194499&11.86&1.11&2741&234675&11.85&1.10&3221&280257&11.85&1.10\\
2333&196577&11.87&1.10&2749&235410&11.84&1.10&3229&281058&11.85&1.10\\
2339&197088&11.86&1.11&2753&235819&11.85&1.10&3251&283128&11.85&1.10\\
2341&197257&11.86&1.11&2767&237102&11.84&1.10&3253&283299&11.84&1.10\\
2347&197855&11.87&1.10&2777&238076&11.85&1.10&3257&283699&11.84&1.10\\
2351&198171&11.85&1.11&2789&239195&11.85&1.10&3259&283870&11.84&1.10\\
2357&198761&11.86&1.11&2791&239404&11.85&1.10&3271&285070&11.85&1.10\\
2371&200085&11.87&1.10&2797&239961&11.85&1.10&3299&287714&11.84&1.10\\
2377&200623&11.86&1.10&2801&240349&11.85&1.10&3301&287860&11.83&1.10\\
2381&200966&11.85&1.11&2803&240517&11.85&1.10&3307&288526&11.85&1.10\\
2383&201190&11.86&1.10&2819&242035&11.85&1.10&3313&289069&11.85&1.10\\
2389&201764&11.87&1.10&2833&243343&11.84&1.10&3319&289688&11.85&1.10\\
2393&202170&11.88&1.10&2837&243688&11.84&1.10&3323&290048&11.85&1.10\\
2399&202630&11.85&1.11&2843&244255&11.84&1.10&3329&290606&11.85&1.10\\
2411&203750&11.85&1.10&2851&245064&11.85&1.10&3331&290799&11.85&1.10\\
2417&204302&11.85&1.11&2857&245584&11.84&1.10&3343&291930&11.84&1.10\\
2423&204921&11.87&1.10&2861&245985&11.84&1.10&3347&292328&11.84&1.10\\
2437&206225&11.86&1.10&2879&247689&11.84&1.10&3359&293499&11.85&1.10\\
2441&206615&11.87&1.10&2887&248512&11.86&1.10&3361&293661&11.84&1.10\\
2447&207126&11.86&1.10&2897&249463&11.86&1.10&3371&294685&11.86&1.10\\
2459&208293&11.87&1.10&2903&249984&11.85&1.10&3373&294881&11.86&1.10\\
2467&209013&11.86&1.10&2909&250556&11.85&1.10&3389&296390&11.85&1.10\\
2473&209553&11.86&1.10&2917&251286&11.84&1.10&3391&296568&11.85&1.10\\
2477&209967&11.87&1.10&2927&252203&11.83&1.10&3407&298091&11.84&1.10\\
2503&212360&11.86&1.10&2939&253398&11.85&1.10&3413&298744&11.86&1.10\\
2521&214013&11.85&1.10&2953&254738&11.85&1.10&3433&300642&11.85&1.10\\
2531&214955&11.85&1.10&2957&255103&11.85&1.10&3449&302186&11.86&1.10\\
2539&215748&11.86&1.10&2963&255591&11.83&1.10&3457&302973&11.86&1.10\\
2543&216116&11.86&1.10&2969&256195&11.84&1.10&3461&303384&11.86&1.10\\
2549&216670&11.86&1.10&2971&256429&11.85&1.10&3463&303522&11.85&1.10\\
2551&216808&11.84&1.10&2999&259079&11.84&1.10&3467&303930&11.86&1.10\\
2557&217425&11.86&1.10&3001&259272&11.84&1.10&3469&304058&11.84&1.10\\
2579&219477&11.86&1.10&3011&260221&11.84&1.10&3491&306176&11.84&1.10\\
2591&220629&11.86&1.10&3019&260999&11.85&1.10&3499&307016&11.86&1.10\\
2593&220756&11.85&1.10&3023&261297&11.83&1.10&3511&308190&11.86&1.10\\
2609&222240&11.84&1.10&3037&262686&11.84&1.10&3517&308731&11.85&1.10\\
2617&223035&11.86&1.10&3041&263047&11.84&1.10&3527&309695&11.85&1.10\\
2621&223399&11.85&1.10&3049&263788&11.83&1.10&3529&309878&11.85&1.10\\
2633&224517&11.85&1.10&3061&264981&11.84&1.10&3533&310275&11.85&1.10\\
2647&225831&11.85&1.10&3067&265596&11.85&1.10&3539&310814&11.85&1.10\\
2657&226765&11.85&1.10&3079&266675&11.84&1.10&3541&311053&11.86&1.10\\
2659&226982&11.86&1.10&3083&267032&11.83&1.10&3547&311652&11.86&1.10\\
2663&227312&11.84&1.10&3089&267643&11.84&1.10&3557&312570&11.85&1.10\\
2671&228116&11.86&1.10&3109&269559&11.85&1.10&3559&312794&11.86&1.10\\
2677&228688&11.86&1.10&3119&270536&11.85&1.10&3571&313944&11.85&1.10\\
2683&229247&11.86&1.10&3121&270727&11.85&1.10&3581&314954&11.86&1.10\\
2687&229626&11.86&1.10&3137&272177&11.83&1.10&3583&315081&11.85&1.10\\
2689&229746&11.84&1.10&3163&274730&11.85&1.10&3593&316089&11.86&1.10\\
2693&230198&11.86&1.10&3167&275065&11.84&1.10&3607&317380&11.85&1.10\\
\midrule
\end{tabular}
\end{footnotesize}
\end{center}
\end{table}

\section{Appendix}

Here we prove certain number-theoretic bounds used in our estimates above.
In most cases stronger asymptotic results are known, but we seek fully explicit bounds.
Our results are derived from bounds for the summatory functions
$$
M(x)=\sum_{1 \le n\le x}\mu(n)\qquad\text{and}\qquad Q(x)=\sum_{1 \le n\le x}|\mu(n)|,
$$
where $\mu(n)$ is the M\"{o}bius function.
The function $Q(x)$ counts squarefree positive integers, and it is well known that $Q(x)\sim x/\zeta(2)=6x/\pi^2$, see \cite{Bateman:NumberTheory} for example.
We let $R(x)=Q(x)-6x/\pi^2$, and note the following result from \cite{Cohen:MobiusBounds}.
\begin{theorem}[Cohen-Dress-El Marraki]\label{CDMthm}
For all $x\ge 2160535$ we have the bound $|M(x)| \le x/4345$, and for all $x\ge 438653$ we have $|R(x)| \le 0.02767\sqrt{x}$.
\end{theorem}
\noindent
It will be more convenient for us to work with slightly weaker bounds.
\begin{corollary}\label{mubounds}
For all $x\ge 10^5$ we have $|M(x)| \le x/900$ and $|R(x)| \le \sqrt{x}/25$.
\end{corollary}
\begin{proof}
By Theorem~\ref{CDMthm}, it suffices to verify the corollary for $x\in[10^5,2160535]$, a task readily accomplished with a machine calculation.
\end{proof}

The proofs below use the Riemann-Stieltjes integral, as defined, for example, in \cite[Ch.~7]{Apostol:Analysis}.
The existence of all the integrals we use, including improper integrals, is easily verified and assumed without further comment.
For a real number $x$, we use $[x]$ to denote the largest integer $n\le x$ and define $\{x\}=x-[x]$.

To make our error terms explicit, for any positive real value $B$ we use $\epsilon(B)$ to denote an unspecified real number with absolute value at most $B$.  Throughout this appendix, all sums range over positive integers.

\begin{lemma}\label{lemma:harmonic}
For all $x \ge 1$ we have
$$
\sum_{n\le x}\frac{1}{n} = \log x + \gamma + \epsilon\left(\frac{1}{2x}+\frac{1}{12x^2}\right)\quad\text{and}\quad
\sum_{n\le x}\frac{|\mu(n)|}{n} = \frac{6}{\pi^2}\log x + \gamma' + \epsilon\left(\frac{3}{25\sqrt{x}}\right),
$$
where $\displaystyle{\gamma = \lim_{x\to\infty}}$$\left(\sum_{n\le x}\frac{1}{n}-\log x\right) \approx 0.577216$
is Euler's constant, and the constant $\gamma'$ is defined by
$\displaystyle{\gamma' = \lim_{x\to\infty}}$$\left(\sum_{n\le x}\frac{|\mu(n)|}{n}-\frac{6}{\pi^2}\log x\right)\approx 1.043895$.
\end{lemma}
\begin{proof}
Both bounds are verified by machine for $x < 10^5$, so we assume $x\ge 10^5$.
The first result is standard, but some care is required to express the bound in terms of~$x$ rather than the integer $[x]$.
From \cite[Eq.~9.89]{Graham:ConcreteMathematics} we have
$$
\sum_{n\le x}\frac{1}{n} \seq H_{[x]} \seq \log [x] + \gamma + \frac{1}{2[x]} - \frac{1}{12[x]^2} + \frac{1}{120[x]^4} - \frac{\theta}{256[x]^6},
$$
where $\theta$ is a positive real number less than 1.  We wish to bound the quantity
\begin{equation}\label{heq1}
\log x + \gamma - H_{[x]} \seq \log\left(1+\frac{\{x\}}{[x]}\right) - \frac{1}{2[x]} + \frac{1}{12[x]^2} - \frac{1}{120[x]^4} + \frac{\theta}{256[x]^6}.\\
\end{equation}
Expanding $\log\left(1+\frac{\{x\}}{[x]}\right)$, the right-hand side of (\ref{heq1}) can be expressed as
\begin{equation}\label{heq2}
\frac{\{x\}}{[x]} - \frac{\{x\}^2}{2[x]^2} + \frac{\{x\}^3}{3[x]^3} - \frac{\{x\}^4}{4[x]^4} - \frac{1}{2[x]} + \frac{1}{12[x]^2} - \frac{1}{120[x]^4} + \epsilon\left(\frac{1}{5[x]^5}\right).
\end{equation}
We now apply $\frac{1}{[x]}=\frac{1}{x}+\frac{\{x\}}{x[x]}$ repeatedly to (\ref{heq2}), which eventually yields
\begin{align*}
\log x + \gamma - H_{[x]} \seq \frac{2\{x\}-1}{2x} &\splus \frac{6\{x\}^2-6\{x\}+1}{12x^2} \splus \frac{2\{x\}^3-3\{x\}^2+\{x\}}{6x^3}\\
&\splus \frac{30\{x\}^4-60\{x\}^3+30\{x\}^2-1}{120x^4} \splus \epsilon\left(\frac{5}{[x]^5}\right).
\end{align*}
Recalling that $x\ge 10^5$, we note that if $\{x\} \le 0.9$ then the sum of the terms on the RHS is less than $1/(2x)$.  When $\{x\} > 0.9$ then the first two terms are both positive, with sum less than $\frac{1}{2x}+\frac{1}{12x^2}$, while the sum of the remaining three terms is negative, with absolute value less than $\frac{1}{2x}+\frac{1}{12x^2}$.  This proves the first result.

To prove the second result, we integrate by parts and apply Corollary \ref{mubounds}:
\begin{align*}
\sum_{n\le x}\frac{|\mu(n)|}{n} &\seq \int_{1^-}^x\frac{1}{t} dQ(t) \seq \frac{Q(t)}{t}\Bigm|_{1^-}^x + \int_{1^-}^x\frac{Q(t)}{t^2}dt\\
&\seq \frac{6}{\pi^2}+\epsilon\left(\frac{1}{25\sqrt{x}}\right)+\int_{1}^x\frac{6}{\pi^2t}dt +\int_1^x\frac{R(t)}{t^2}dt\\
&\seq \frac{6}{\pi^2}\log x + \left(\frac{6}{\pi^2} + \int_1^\infty\frac{R(t)}{t^2}dt\right) - \int_x^\infty\frac{R(t)}{t^2}dt + \epsilon\left(\frac{1}{25\sqrt{x}}\right).
\end{align*}
Subtracting $\frac{6}{\pi^2}\log x$ from both sides and taking the limit as $x\to\infty$ yields
$$
\left(\frac{6}{\pi^2} + \int_{1}^\infty\frac{R(t)}{t^2}dt\right) = \lim_{x\to\infty}\left(\sum_{n\le x}\frac{|\mu(n)|}{n}-\frac{6}{\pi^2}\log x\right) = \gamma'.
$$
Corollary \ref{mubounds} implies $\int_x^\infty\frac{R(t)}{t^2}dt=\epsilon\left(\frac{2}{25\sqrt{x}}\right)$, and the total error is  $\epsilon\left(\frac{3}{25\sqrt{x}}\right)$.
\end{proof}

\begin{lemma}\label{lemma:musums}
For all $x\ge 10^5$ we have
$$
\sum_{n\le x}\frac{\mu(n)}{n^2} \seq \frac{6}{\pi^2} + \epsilon\left(\frac{1}{300x}\right)\qquad\text{and}\qquad
\sum_{n\le x}\frac{\mu(n)\log n}{n^2} \seq C + \epsilon\left(\frac{3\log x+1}{900x}\right),
$$
where $C=\sum\frac{\mu(n)\log n}{n^2}\approx -0.346495$.
\end{lemma}
\begin{proof}
For the first bound, we recall that $\sum\mu(n)/n^2 = 1/\zeta(2)= 6/\pi^2$.   Therefore
$$
\sum_{n\le x}\frac{\mu(n)}{n^2}\seq \sum\frac{\mu(n)}{n^2} - \sum_{n>x} \frac{\mu(n)}{n^2}\seq \frac{6}{\pi^2} - \sum_{n>x} \frac{\mu(n)}{n^2}.
$$
To bound the tail, we integrate by parts and apply Corollary \ref{mubounds}:
\begin{align*}
\sum_{n > x}\frac{\mu(n)}{n^2} &\seq\int_x^\infty \frac{1}{t^2}dM(t) \seq \frac{M(t)}{t^2}\Bigm|_x^\infty - \int_x^\infty M(t)d\frac{1}{t^2}\\
&\seq \frac{-M(x)}{x^2} + 2\int_x^\infty \frac{M(t)}{t^3}dt\seq \epsilon\left(\frac{1}{300x}\right).
\end{align*}
The second bound is proved similarly.  When bounding the tail one uses
\begin{align*}
\sum_{n > x}\frac{\mu(n)\log n}{n^2}&\seq \frac{-M(x)\log x}{x^2} \splus \int_x^\infty\frac{M(t)(2\log t - 1)}{t^3}dt\\
&\seq \epsilon\left(\frac{\log x}{900x}\right) \splus \epsilon\left(\frac{\log x}{450x}+\frac{1}{900x}\right)\seq \epsilon\left(\frac{3\log x+1}{900x}\right)
\end{align*}
to complete the proof.
\end{proof}

\begin{lemma}\label{lemma:A1}
For all $x\ge 1$ we have
$$
\sum_{n\le x} \frac{\phi(n)}{n^2} \seq \frac{6\log x}{\pi^2} \splus \frac{6\gamma}{\pi^2} - C \splus \epsilon\left(\frac{\log x+1}{3x}\right),
$$
where $\gamma \approx 0.577216$ is Euler's constant and $C=\sum\frac{\mu(n)\log n}{n^2}\approx -0.346495$.
\end{lemma}
\begin{proof}
The lemma is machine verified for $x < 10^5$, so we assume $x \ge 10^5$.
We recall that $\phi(n)=\sum_{d|n}\mu(d)\frac{n}{d}$, and therefore
$$
\sum_{n\le x}\frac{\phi(n)}{n^2} \seq \sum_{n\le x}\frac{1}{n}\sum_{d|n}\frac{\mu(d)}{d}\seq
\sum_{d\le x}\frac{\mu(d)}{d}\sum_{qd\le x} \frac{1}{qd} \seq \sum_{d\le x}\frac{\mu(d)}{d^2}\sum_{q\le x/d}\frac{1}{q}.
$$
By Lemma~\ref{lemma:harmonic} we have $\sum_{q\le x/d}\frac{1}{q} = \log(x/d)+\gamma+\epsilon\left(\frac{d}{2x}+\frac{d^2}{12x^2}\right)$, which yields
$$
\sum_{n\le x}\frac{\phi(n)}{n^2}\seq \bigl(\log x + \gamma\bigr) \sum_{d\le x}\frac{\mu(d)}{d^2}\sminus\sum_{d\le x}\frac{\mu(d)\log d}{d^2}\splus \epsilon\left(\frac{1}{2x}\right)\sum_{d\le x}\frac{|\mu(d)|}{d}\splus\epsilon\left(\frac{1}{12x^2}\right)Q(x).
$$
Applying Lemma~\ref{lemma:musums} to the first two sums and Lemma~\ref{lemma:harmonic} to the third sum yields
\begin{align*}
\sum_{n\le x}\frac{\phi(n)}{n^2}\seq &\frac{6\log n}{\pi^2} \splus \frac{6\gamma}{\pi^2}-C\\\notag
&\splus \epsilon\left(\left(\frac{1}{150}+\frac{3}{\pi^2}\right)\frac{\log x}{x}+\left(\frac{3\gamma+1}{900}+\frac{\gamma'}{2}+\frac{1}{2\pi^2}\right)\frac{1}{x}+\frac{19}{300x^{3/2}}\right).
\end{align*}
For $x \ge 10^5$ we can replace the error term by $\epsilon\left(\frac{\log x+1}{3x}\right)$.
\end{proof}

\begin{lemma}\label{lemma:A2}
For all $x \ge 3/2$ we have
$$
\sum_{n\le x}\phi(n) \seq \frac{3x^2}{\pi^2}\splus \epsilon\left(\frac{x\log x}{2}\right).
$$
\end{lemma}
\begin{proof}
The lemma is machine verified for $x<10^5$, so we assume $x\ge 10^5$.
As above, we apply $\phi(n)=\sum_{d|n}\mu(d)\frac{n}{d}$ to obtain
\begin{equation}\label{A2eq1}
\sum_{n\le x}\phi(n)\seq \sum_{n\le x}\sum_{d|n}\mu(d)\frac{n}{d} \seq \sum_{d\le x}\mu(d)\sum_{qd\le x}q \seq \frac{1}{2}\sum_{d\le x}\mu(d)\left(\left[\frac{x}{d}\right]^2+\left[\frac{x}{d}\right]\right).
\end{equation}
We note that $\sum_{d\le x}\mu(d)\left[\frac{x}{d}\right]=1$, by \cite[Thm.~3.12]{Apostol:NumberTheory}, and we have
\begin{align}\label{A2eq2}
\sum_{d\le x}\mu(d)\left[\frac{x}{d}\right]^2 &\seq \sum_{d\le x}\mu(d)\left(\frac{x}{d}-\left\{\frac{x}{d}\right\}\right)^2\\\notag
&\seq \sum_{d\le x}\mu(d)\left(\frac{x}{d}\right)^2\sminus 2\sum_{d\le x}\mu(d)\left\{\frac{x}{d}\right\}\left(\frac{x}{d}\right)\splus\sum_{d\le x}\mu(d)\left\{\frac{x}{d}\right\}^2\\\notag
&\seq x^2\sum_{d\le x}\frac{\mu(d)}{d^2}\sminus 2x\sum_{d\le x}\left\{\frac{x}{d}\right\}\frac{\mu(d)}{d} \splus \epsilon(x).
\end{align}
This first sum is addressed by Lemma~\ref{lemma:musums}.  For the second we use
\begin{align*}
\sum_{d\le x}\left\{\frac{x}{d}\right\}\frac{\mu(d)}{d} &\seq \frac{1}{2}\sum_{d\le x}\frac{\mu(d)}{d} \splus \sum_{d\le x}\left(\left\{\frac{x}{d}\right\}-\frac{1}{2}\right)\frac{\mu(d)}{d}\\
&\seq\frac{1}{2}\sum_{d\le x}\frac{\mu(d)}{d}\splus\epsilon\left(\frac{1}{2}\right)\sum_{d\le x}\frac{|\mu(d)|}{d}.
\end{align*}
By \cite[Lemma~3.2]{Bateman:NumberTheory}, the first sum is $\epsilon(1+\frac{1}{x})$.  Applying Lemma~\ref{lemma:harmonic} we obtain
$$
\sum_{d\le x}\left\{\frac{x}{d}\right\}\frac{\mu(d)}{d} \seq \epsilon\left(\frac{3\log x}{\pi^2} + \frac{\gamma'+1}{2} + \frac{3}{50\sqrt{x}}+\frac{1}{2x}\right).
$$
Since $x \ge 5$, the error term is $\epsilon\left(\frac{3\log x}{\pi^2} + \frac{21}{20}\right)$.  Applying this to (\ref{A2eq2}) yields
$$
\sum_{d\le x}\mu(d)\left[\frac{x}{d}\right]^2\seq \frac{6x^2}{\pi^2} \splus \epsilon\left(\frac{x}{300}\right) + \epsilon\left(\frac{6x\log x}{\pi^2}+\frac{21x}{10}\right) \splus \epsilon(x),
$$
and from (\ref{A2eq1}) we then have
$$
\sum_{n\le x}\phi(n) \seq \frac{3x^2}{\pi^2} \splus \epsilon\left(\frac{3x\log x}{\pi^2} + \frac{931x}{600} + \frac{1}{2}\right).
$$
Since $x\ge 3000$, we may replace the error term with $\epsilon\left(\frac{1}{2}x\log x\right)$.
\end{proof}

\begin{corollary}\label{cor:A3}
Let $K(x) = \sum_{n\le x}\phi(n)$.  For all $x\ge 1$ we have the bounds
$$K(C_1x)\ge \text{\rm \textonequarter}K(x)-2,\quad K(C_2x)\ge \text{\rm \textonehalf}K(x)-2,\quad K(C_3x)\ge \text{\rm \textthreequarters}K(x)-2,$$
where $C_1=0.539$, $C_2=0.742$, and $C_3=0.917$.  
\end{corollary}
\begin{proof}
We verify by machine for $x< 300$ and apply Lemma~\ref{lemma:A2} for $x\ge 300$.
\end{proof}

\begin{lemma}\label{lemma:heightfromroots}
Let $P(X)=\sum c_kX^k$ be a monic polynomial in $\C[X]$ of degree $n$ whose roots $\omega_j$ satisfy $|\omega_1|\ge\cdots\ge|\omega_n|> 1$.  Let $m=|\omega_n|$ and $M=\prod|\omega_j|$.  Then
$$
\log |c_k| \sle \log M + \left(\frac{\log m + 1}{m}\right)n.
$$
\end{lemma}
\begin{proof}
Let $M_k=\prod_{j\le k}|\omega_j|$ for $0\le k\le n$.
If $S$ is any $k$-subset of $\{1,\ldots,n\}$ then we have $\prod_{j\in S}|\omega_j|\le M_k$.  Therefore
\begin{equation}\label{ckbound}
|c_k|\le \binom{n}{k}M_{n-k}.
\end{equation}
It suffices to bound the right-hand side of (\ref{ckbound}), and we need only consider $k\le n/2$.  Now suppose $k \ge (n+1)/(m+1)$, with $1\le k\le n/2$.  We then have
$$
\binom{n}{k}\Bigm/\binom{n}{k-1} = \frac{n-k+1}{k} \le m,
$$
which implies $\binom{n}{k}M_{n-k} \le \binom{n}{k-1}M_{n-k+1}$.  Thus we may assume $k < (n+1) / (m+1)$.
The lemma clearly holds for $k=0$, so we assume $m < n$ and set $r=[(n+1)/(m+1)]$.  We then have
$$
|c_k|\le \binom{n}{r}M,
$$
and it remains only to show that $\log\binom{n}{r} \le \frac{n}{m}(\log m+ 1)$.  Using the explicit bounds
$$
\sqrt{2\pi}n^{n-\frac{1}{2}} e^{-n+\frac{1}{12n+1}}\slt n! \slt \sqrt{2\pi}n^{n-\frac{1}{2}} e^{-n+\frac{1}{12n}}
$$
for Stirling's formula \cite{Robbins:StirlingFormula}, one obtains the inequality
\begin{equation}\label{binombound}
\binom{n}{k} \slt \frac{n^n}{k^k(n-k)^{n-k}},
\end{equation}
valid for $0<k<n$.  For fixed $n$ and $k\le n/2$ the right-hand side of (\ref{binombound}) is an increasing function of $k$.  We thus obtain
$$
\binom{n}{r}\slt \frac{n^n}{(n/m)^{n/m}(n-n/m)^{n-n/m}} \seq \frac{m^n}{(m-1)^{n-n/m}},
$$
using $r \le (n+1)/(m+1) < n/m$ (since $m < n$).  Taking logarithms, we have
\begin{align*}
\log \binom{n}{r} &\slt n\log m - \left(n-n/m\right)\log(m-1)\\
&\seq n\bigl(\log m - (1-1/m)(\log m + \log(1-1/m))\bigr)\\
&\seq n\bigl((1/m)\log m - (1-1/m)\log(1-1/m)\bigr)\\
&\seq \frac{n}{m}\left(\log m + (m-1)\log\left(1+\frac{1}{m-1}\right)\right)\\
&\sle \frac{n}{m}\bigl(\log m + 1\bigr),
\end{align*}
which completes the proof.
\end{proof}

\begin{lemma}\label{lemma:A4}
Let $P\in\C[X,Y]$ be a nonzero polynomial of degree at most $n$ in each variable.  Suppose $h(P(X,y))\le B$ for all real $y$ in the interval $[L,2L]$, for some real $B>0$ and $L>1$.  Then we have
$$
h(P)\sle B + \left(\frac{\log L + 1}{L} + 3\log 2\right)n.
$$
\end{lemma}
\begin{proof}
Let $P(X,Y)=\sum Q_m(Y)X^m$.  It suffices to show $h(Q_m)$ satisfies the bound in the lemma for each nonzero $Q_m$.
For any $y\in\C$, the coefficient of $X^m$ in $P(X,y)$ is $Q_m(y)$.
As in \cite[Lemma~9]{CohenPaula:ModularPolynomials}, let us pick $n+1$ interpolation points $y_k\in [L,2L]$:
$$
y_k=L\left(1+\frac{k}{n}\right)\qquad\qquad(0\le k\le n).
$$
We may interpolate $Q_m(Y)$ from the pairs $(y_k,c_{k,m})$, where $c_{k,m}$ is the coefficient of $X^m$ in $P(X,y_k)$.
Applying the Lagrange interpolation formula yields
$$
Q_m(Y) = \sum_{k=0}^n c_{k,m}\prod_{j\ne k}\frac{Y-y_j}{y_k-y_j}.
$$
By Lemma~\ref{lemma:heightfromroots}, the coefficients of $\prod_{j\ne k}(Y-y_j)$ have absolute values bounded by $\delta M$, where $\delta = \exp\bigl((n/L)(\log L + 1)\bigr)$ and
$$
M = \prod_{j\ne k}y_j \le \frac{L^n(2n)!}{n^n n!}.
$$
We also have
$$
\prod_{j\ne k}|y_k-y_j| \seq L^n\frac{k}{n}\cdot\frac{k-1}{n}\cdots\frac{1}{n}\cdot\frac{1}{n}\cdot\frac{2}{n}\cdots\frac{n-k}{n}\seq\frac{L^n k!(n-k)!}{n^n},
$$
and by the hypothesis of the lemma, $|c_{k,m}|< C$, where $C=e^B$.  Thus the absolute values of the coefficients of $Q_m(Y)$ are bounded by
$$
\sum_{k=0}^n C \frac{\delta (2n)!}{n!k!(n-k)!} \seq \delta C\binom{2n}{n}\sum_{k=0}^n\binom{n}{k} \sle 2^{3n}\delta C.
$$
Assuming $Q_m(Y)$ is nonzero, we take logarithms and obtain
$$
h(Q_m)\sle B + \left(\frac{\log L + 1}{L} + 3\log 2\right)n.
$$
as desired.
\end{proof}

\bibliographystyle{amsplain}
\providecommand{\bysame}{\leavevmode\hbox to3em{\hrulefill}\thinspace}
\providecommand{\MR}{\relax\ifhmode\unskip\space\fi MR }
\providecommand{\MRhref}[2]{%
  \href{http://www.ams.org/mathscinet-getitem?mr=#1}{#2}
}
\providecommand{\href}[2]{#2}


\begin{thebibliography}{10}

\bibitem{Anderson:RamanujanSum}
Douglas~R. Anderson and Tom~M. Apostol, \emph{The evaluation of {R}amanujan's
  sum and generalizations}, Duke Mathematics Journal \textbf{11} (1953),
  211--216.

\bibitem{Apostol:Analysis}
Tom~M. Apostol, \emph{Mathematical analysis}, second ed., Addison-Wesley, 1974.

\bibitem{Apostol:NumberTheory}
\bysame, \emph{Introduction to analytic number theory}, Springer, 1976.

\bibitem{Bateman:NumberTheory}
Paul~T. Bateman and Harold~G. Diamond, \emph{Analytic number theory}, World
  Scientific, 2004.

\bibitem{Brisebarre:jFunctionCoefficients}
Nicolas Brisebarre and Georges Philibert, \emph{Effective lower and upper
  bounds for the {F}ourier coefficients of powers of the modular invariant
  $j$}, The Ramanujan mathematical society \textbf{20} (2005), no.~4, 255--282.

\bibitem{BrokerLauterSutherland:CRTModPoly}
Reinier Br{\"o}ker, Kristin Lauter, and Andrew~V. Sutherland, \emph{Modular
  polynomials via isogeny volcanoes}, 2009,
  \url{http://arxiv.org/abs/1001.0402}.

\bibitem{CharlesLauter:ModPoly}
Denis Charles and Kristin Lauter, \emph{Computing modular polynomials}, LMS
  Journal of Computation and Mathematics \textbf{8} (2005), 195--204.

\bibitem{Cohen:MobiusBounds}
Henri Cohen, Francois Dress, and Mohamed~El Marraki, \emph{Explicit estimates
  for summatory functions linked to the {M}\"{o}bius $\mu$-function.},
  Functiones et Approximatio \textbf{37} (2007), 51--63.

\bibitem{CohenPaula:ModularPolynomials}
Paula Cohen, \emph{On the coefficients of the transformation polynomials for
  the elliptic modular function}, Math. Proc. of the Cambridge Philosophical
  Society \textbf{95} (1984), 389--402.

\bibitem{Elkies:AtkinBirthday}
Noam~D. Elkies, \emph{Elliptic and modular curves over finite fields and
  related computational issues}, Computational Perspectives on Number Theory
  (D.~A. Buell and J.~T. Teitelbaum, eds.), Studies in Advanced Mathematics,
  vol.~7, AMS, 1998, pp.~21--76.

\bibitem{Enge:ModularPolynomials}
Andreas Enge, \emph{Computing modular polynomials in quasi-linear time},
  Mathematics of Computation \textbf{78} (2009), 1809--1824.

\bibitem{Graham:ConcreteMathematics}
Ronald~L. Graham, Donald~E. Knuth, and Oren Patashnik, \emph{Concrete
  mathematics: A foundation for computer science}, second ed., Addison-Wesley,
  1994.

\bibitem{Hardy:NumberTheory}
Godfrey~H. Hardy and Edward~M. Wright, \emph{An introduction to the theory of
  numbers}, fifth ed., Oxford Science Publications, 1979.

\bibitem{Holder:RamanujanSum}
Otto H\"older, \emph{Zur {T}heorie der {K}reisteilungsgleichung $k_m(x)=0$},
  Prace Matematyczno Fizyczne \textbf{43} (1936), 13--23.

\bibitem{Lang:EllipticFunctions}
Serge Lang, \emph{Elliptic functions}, second ed., Springer-Verlag, 1987.

\bibitem{Robbins:StirlingFormula}
Herbert Robbins, \emph{A remark on {S}tirling's formula}, American Mathematical
  Monthly \textbf{62} (1955), 26--29.

\bibitem{Walfisz:PhiSumBound}
Arnold Walfisz, \emph{Weylsche {E}xponentialsummen in der neueren
  {Z}ahlentheorie}, Mathematische Forschungsberichte, XV, VEB Deutscher Verlag
  der Wissenschaften, 1963.

\end{thebibliography}
\end{document}